\documentclass[reqno]{amsart}
\usepackage{amssymb,amsmath,amsthm,amsfonts,fancyhdr}
\usepackage{indentfirst}
\usepackage{mathrsfs}
\usepackage{setspace}
\usepackage{color}
\usepackage{cite}
\usepackage{bm}
\usepackage{lineno}
\usepackage{latexsym}
\usepackage{verbatim}
\usepackage{cases}
\usepackage{esint}
\theoremstyle{plain} 
\newtheorem{thm}{Theorem}[section]

\newtheorem*{thm*}{Theorem}

\theoremstyle{definition}

\theoremstyle{remark}
\newtheorem{rmk}[thm]{Remark}

\renewcommand{\det}{\mbox{det}}

  \numberwithin{equation}{section}
  \numberwithin{figure}{section}

\begin{document}

\title{\textbf{ The asymptotic behavior for divergence elliptic equations in  exterior domains with   periodic coefficients }}

\author{Lichun Liang}

\address{School of Mathematical Sciences, Chongqing Normal University, Chongqing, 401331, P.R. China.}

\email{lianglichun@cqnu.edu.cn}

\begin{abstract}
In this paper,  we  investigate the asymptotic behavior of solutions for
divergence linear elliptic equations in exterior domains with periodic coefficients.  Consequently, we generalise the Liouville type result firstly established by Avellaneda and Lin.
\end{abstract}

\keywords{Linear Elliptic Equations; Exterior Domains; Asymptotic Behavior; Periodic Coefficients}
\date{}
\maketitle

\section{Introduction}

\noindent

In this paper, we are concerned with
the asymptotic behavior of solutions for divergence linear elliptic equations in  exterior domains with periodic coefficients. More precisely, we consider linear uniformly elliptic equations of the
divergence form

\begin{equation}\label{eq1}
  L u=D_i(a_{ij}(x)D_ju(x))=0,\ \ \ x\in \mathbb{R}^n\setminus \overline{B_1}.
\end{equation}
We assume that the coefficients $a_{ij}$ satisfy the following three conditions:

\emph{(1).} (\textbf{Symmetry})
$$a_{ij}=a_{ji};$$

\emph{(2).} (\textbf{Periodicity})
$$a_{ij}(x+z)=a_{ij}(x)\ \   \mbox{for all}\ \ \ x\in \mathbb{R}^n\ \ \ \mbox{and}\ \ \ z\in \mathbb{Z}^n;$$

\emph{(3).} (\textbf{Ellipticity})
There exist two constants $0<\lambda \leq \Lambda <\infty$ such that
      $$\lambda|\xi|^2\leq a_{ij}(x)\xi_i \xi_j \leq \Lambda |\xi|^2$$
for any $\xi=(\xi_1,\cdots, \xi_n)\in \mathbb{R}^n$.

In order to expound  the main motivations on this work, we would like to confine our attention to the development of Liouville type results for elliptic equations with the periodic data in the whole space. In  1989, by implementing  the tools from homogenization theory, Avellaneda and Lin \cite{AL1989}  first obtained a Liouville type result for  linear elliptic equations of the divergence form
$\partial_i(a_{ij}(x)\partial_ju(x))=0$ in $\mathbb{R}^n$ with the periodic data.
Under the hypothesis that the coefficients $a_{ij}(x)$ are Lipschitz continuous and periodic, they showed that any polynomial growth solution of degree of at most $m$ must be a polynomial  with  periodic coefficients.  A few years later, Moser and Struwe \cite{MS92} considered quasilinear elliptic equations $-div (F_{p}(x,Du(x)))=0$ in $\mathbb{R}^n$, where $F(x,p)$ is periodic in $x$ and satisfies  convexity and suitable growth assumptions with respect to $p$. Using the Harnack inequality from the elliptic equation theory, they showed that any linear growth solution must be a linear function up to a periodic perturbation, which partially generalizes Avellaneda and Lin's results from the linear to the nonlinear case. Moreover, they also  achieved a simplified proof for the linear case without the Lipschitz continuous assumption on the coefficients.  For linear elliptic equations of the non-divergence form $a_{ij}(x)D_{ij}u(x)=0$ in $\mathbb{R}^n$ with measurable and periodic coefficients, Li and Wang \cite{LW01}  proved a result similar to that in \cite{AL1989}.  For Monge-Amp\`{e}re equations,
Caffarelli and Li \cite{CL04}  showed that any convex solution $u$ of $\det(D^2u(x))=f(x)$ in  $\mathbb{R}^n$ must be a quadratic polynomial  up to a periodic perturbation under the condition that $f$ is periodic and smooth. Recently, the smooth assumption on $f$ has been weakened into $f\in L^{\infty}(\mathbb{R}^n)$ by Li and Lu\cite{LL22}. Considering fully nonlinear
uniformlly elliptic equations $F(D^2u)=f$ in  $\mathbb{R}^n$ with the periodic hand term $f$, Li and Liang \cite{LL2024} established the same conclusion as  Monge-Amp\`{e}re equations under the condition that the solution grows no faster than $|x|^2$ at infinity.

Now we turn to asymptotic results for elliptic equations in exterior domains.  For linear elliptic equations, it was shown by  Gilbarg and  Serrin \cite{GS55} that any  positive solution $u$ of $a_{ij}(x)D_{ij}u(x)=0$ in $\mathbb{R}^n \setminus \overline{B}_1$ $(n\geq 3)$ must be of the form $u=u_{\infty}+O(|x|^{2-n})$ under the condition that the coefficients $a_{ij}$ are Dini  continuous at infinity, where $u_{\infty}$ is a constant. 
Analogously, Serrin \cite{S1965} established the same result as Gilbarg and  Serrin's ones for  positive solutions of $\partial_i(a_{ij}(x)\partial_ju(x))=0$ in $\mathbb{R}^n \setminus \overline{B}_1$ $(n\geq 3)$. Furthermore, Serrin and Weinberger \cite{SW1966} have improved Serrin's asymptotic result into
$u=u_{\infty}+a |x|^{2-n} +O(|x|^{2-n-\delta})$ via  the Kelvin transformation, where $a$ is  a constant and the constant $\delta>0$ depends only on ellipticity constants $\lambda$ and $\Lambda$.
As for Monge-Amp\`{e}re equations $\det(D^2u)=1$ in exterior domains, the asymptotic behavior at infinity was established by Caffarelli and Li \cite{CL03}, in which we can also refer to the detailed references for $n=2$. As a generalization of Caffarelli and Li's result, Teixeira and Zhang \cite{TZ16} investigated the case that the right hand is a perturbation of a periodic function.
Finally, we pay attention to fully nonlinear uniformly elliptic equations $F(D^2u)=f$ in exterior domains.
When the right hand $f$ is a constant, Li, Li and Yuan \cite{LLY20} obtained the asymptotic behavior at infinity under the smooth assumption on $F$ for the high dimensional case, whereas the two dimensional case was solved by Li and Liu \cite{LL24}.
Subsequently, Lian and Zhang \cite{LZP} obtained an asymptotic  result  without the smooth assumption on $F$. When the right hand $f$ is a periodic function, Li and Liang \cite{LL2024} established an asymptotic result.

In this paper, our purpose is to investigate the asymptotic behavior of solutions of (\ref{eq1}), that is, we  will generalise the Liouville type result firstly established by Avellaneda and Lin \cite{AL1989}.  Let us  recall Avellaneda and Lin's result as follows.
\bigskip
\begin{thm*}[Avellaneda-Lin]
Assume that the coefficients $a_{ij}$  satisfy (1)-(3) and are Lipschitz continuity. Let $u\in W_{loc}^{1,2}(\mathbb{R}^n)$ be a solution of $$D_i(a_{ij}D_ju)=0\ \ \mbox{in}\ \ \mathbb{R}^n$$
with $$|B_r|^{-1} \int_{B_r}u^2\,dx \leq Cr^{2N}$$
for all $r \geq 1$ and some  integer $N\geq 0$ and constant $C>0$.
Then $$u(x)=\sum_{|\nu|\leq N} p_\nu(x) x^{\nu},$$
where $p_\nu(x)$ are periodic and H\"{o}lder continuous. Moreover, when $|\nu|=N$, the coefficients  $p_\nu(x)$ are constants.
\end{thm*}
\bigskip
\begin{rmk}
For the case $N=1$, Moser and Struwe \cite{MS92} not only removed the Lipschitz continuity assumption on the coefficients $a_{ij}$
but also gave a simple proof of Theorem (Avellaneda-Lin).
\end{rmk}
\bigskip
Before stating  our main results, we introduce the space $\mathbb{T}$ consisting of all continuously periodic functions with zero mean, defined by
$$\mathbb{T}=\left\{v\in C(\mathbb{R}^n):v(x+z)=v(x)\ \mbox{for all}\ x\in \mathbb{R}^n\ \mbox{and}\ z\in \mathbb{Z}^n, \fint_{[0,1]^n}v\,dx=0\right\}.$$
Moreover, let $Q_1=[-\frac{1}{2},\frac{1}{2}]^n$.

As an extension of Liouville theorem, we investigate
the asymptotic behavior at infinity of solutions of (\ref{eq1}) in exterior domains.
\bigskip
\begin{thm}\label{th1}
Let $n\geq 3$ and the coefficients $a_{ij}$  satisfy (1)-(3) and are Lipschitz continuity. Let $u\in W_{loc}^{1,2}(\mathbb{R}^n \setminus \overline{B_1})$ be a solution of $$D_i(a_{ij}D_ju)=0\ \ \mbox{in}\ \ \mathbb{R}^n\setminus \overline{B_1}$$
with $$|B_r|^{-1} \int_{B_r\setminus \overline{B}_1}u^2\,dx \leq Cr^{2N}$$
for all $r > 1$ and some  integer $N\geq 0$ and constant $C>0$.
Then $$u(x)=\sum_{|\nu|\leq N} p_\nu(x) x^{\nu}+a |x|^{2-n} +O(|x|^{2-n-\delta})\ \ \ \mbox{as}\ \ \ |x|\rightarrow \infty,$$
where $p_\nu(x) $ are   periodic and H\"{o}lder continuous, $a$ is  a constant and the constant $\delta>0$ depends only on ellipticity constants $\lambda$ and $\Lambda$.
 Moreover, when $|\nu|=N$, the coefficients  $p_\nu(x)$ are constants.
\end{thm}
\bigskip
Given the asymptotic behavior with linear growth at infinity,  we can establish the following existence theorem for the Dirichlet problem on exterior domains.
\bigskip
\begin{thm}\label{th2}
Let $n\geq 3$ and the coefficients $a_{ij}$  satisfy (1)-(3). Assume that $\Omega\subset \mathbb{R}^n$ is a domain satisfying an exterior cone condition and
$b\in \mathbb{R}^n$ satisfies $\fint_{\partial Q_1}a_{ij}(x)b_i\nu_j\,dx=0$, where $\nu=(\nu_1,\cdots,\nu_n)$ is the unit outward normal to $ \partial Q_1$.
Then for any  $\varphi\in C(\partial \Omega)$, there exist some  $c\in \mathbb{R}$ and $v\in \mathbb{T}$ such that
there exists a   solution $u\in W^{1,2}_{loc}(\mathbb{R}^n \backslash \overline{\Omega})\cup C(\mathbb{R}^n \backslash \overline{\Omega})$  of
\begin{equation*}
   \left\{
\begin{aligned}
   &D_i(a_{ij}D_ju)=0\ \ \mbox{in}\ \ \mathbb{R}^n \backslash \overline{\Omega}\\
&u=\varphi\ \  \mbox{on}\ \ \partial \Omega
\end{aligned}
\right.
 \end{equation*}
satisfying
$$u(x)-b\cdot x-c-v(x)=O(|x|^{2-n})\ \ \ \mbox{as}\ \ \ |x|\rightarrow \infty,$$
where $v\in \mathbb{T}\cap W^{1,2}(\mathbb{R}^n)$ is a unique weak solution of  $D_{i}(a_{ij}D_jv)=-D_{i}(a_{ij}b_j)$ in $\mathbb{R}^n$.
\end{thm}
\bigskip

\begin{rmk}
In Theorem \ref{th1}, even though $N=1$, we cannot dispense with the Lipschitz continuity assumption on the coefficients $a_{ij}$
 for the reason that the extension of the solution $u$ in $\mathbb{R}^n$ need to become a solution of a linear elliptic equation.
\end{rmk}

Now let us illustrate the main ideas in the proof of Theorem \ref{th1} and \ref{th2}.  For the proof of Theorem \ref{th1},
we construct a family of weak solutions $u_r$ for  the Dirichlet problem of $D_i(a_{ij}D_ju)=0$ in $B_r$ with $u_r=u$ on $\partial B_r$. Then by the aid of the linear elliptic equation for $\bar{u}$ that is an extension of the solution $u$ in $W^{2,2}(\mathbb{R}^n)$, we will show that a subsequence of $\{u_r\}_{r=1}^\infty$ locally converges uniformly to a function $v\in C(\mathbb{R}^n)\cap W^{1,2}_{loc}(\mathbb{R}^n)$ that is an entire solution of $D_i(a_{ij}D_jv)=0$. Next we know that $u-v$ is bounded and satisfies $D_i(a_{ij}D_j(u-v))=0$ in $\mathbb{R}^n \setminus \overline{B}_1$. Finally, by the classical asymptotic result for linear elliptic equations and Avellaneda-Lin's result, the proof of Theorem \ref{th1} is completed. As for the proof of Theorem \ref{th2}, we can have a solution $u$ of $D_i(a_{ij}D_ju)=0$ in  $\mathbb{R}^n \backslash \overline{\Omega}$ by solving the Dirichlet problem of $D_i(a_{ij}D_ju_r)=0$ in $\ B_r \backslash \overline{\Omega}$. Then by a barrier function, it is shown that $u$ is continuous up to $\partial \Omega$. Finally, using the Harnack inequality and  the comparison principle, we complete the proof Theorem \ref{th2}.

\section{The proof of Theorem \ref{th1}}

\noindent

\begin{proof}[Proof of Theorem \ref{th1}]
Without loss of generality, by the extension theorem for Sobolev spaces and the regularity for weak solutions, we assume $u\in W^{2,2}(\mathbb{R}^n)$.

For $r>2$, by the Lax-Milgram theorem, we let $u_r\in W^{1,2}(B_r)\cap C(B_r)$  be a weak solution of
\begin{equation*}
   \left\{
\begin{aligned}
   &D_{i}(a_{ij}D_ju_r)=0\ \ \mbox{in}\ \ B_r,\\
&u_r=u\ \  \mbox{on}\ \ \partial B_r.
\end{aligned}
\right.
 \end{equation*}
We will show that $\{u_r\}_{r >2}$ are  bounded in compact sets  of $\mathbb{R}^n$, i.e.,
 $$\|u_r\|_{L^{\infty}(K)}\leq C_{K}$$
 for any compact set $K\subset \mathbb{R}^n$ and some constant $C_{K}>0$ independent of $r$. For this purpose, we let
 $$\bar{u}=(1-\eta)u+\eta,$$
 where $0\leq \eta \leq 1$ be in $C^{\infty}(\mathbb{R}^n)$ with compact support in $B_2$ such that $\eta=1$ in $B_1$. Then  $\bar{u}\in W^{2,2}(\mathbb{R}^n)$ keeps $\bar{u}=u$ outside $B_2$ and we
 set $$D_i(a_{ij}D_j\bar{u})=f\ \ \ \mbox{in}\ \ \ \mathbb{R}^n,$$
 where $f\in L^2(\mathbb{R}^n)$ has support in $B_2$.  
Now we recall a result of Littman, Stampacchia and Weinberger\cite{LSW1963}, which asserts the existence of a Green's $G(x,y)$ for  $L$ in $\mathbb{R}^n$ satisfying
$$-D_{y_i}(a_{ij}D_{y_j}G(x,y))=\delta_x$$
and
 $$C^{-1}|x-y|^{2-n}\leq G(x,y)\leq C|x-y|^{2-n}$$
 for a suitable constant $C>0$ and all $x,y \in \mathbb{R}^n$. For simplicity, we set $$E(x)=G(x,0).$$
Since $f\in L^2(\mathbb{R}^n)$ has compact support, we can  define
 $$w(x)=\int_{\mathbb{R}^n}G(x,y)|f(y)|\,dy$$
 and therefore we have
 $$D_i(a_{ij}D_jw)=-|f|\ \ \ \mbox{in}\ \ \ \mathbb{R}^n$$
 with  $w(x)=O(|x|^{2-n})$ as $|x|\rightarrow \infty$.

 Since $\bar{u}-u_r$ and $w$ are continuous in $B_1$, there exist constants $0<\varepsilon<1$ and  $\bar{C}>1$ such that
 $$|\bar{u}(x)-u_r(x)| \leq \bar{C} E(x)+w(x),\ \ \ x\in \partial B_\varepsilon \cup \partial B_r.$$
 In $ B_r\backslash B_\varepsilon$, we obtain
 $$D_i(a_{ij}D_j(\bar{C} E+w)) \leq D_i(a_{ij}D_j(u_r-\bar{u})) =-f\leq  |f|=-D_i(a_{ij}D_j(\bar{C} E+w)).$$
 From the comparison principle for weak solutions, it follows that
 \begin{equation}\label{eq2}
    \bar{u}(x)-(\bar{C} E(x)+w(x))\leq u_r(x)\leq \bar{u}(x)+\bar{C} E(x)+w(x),\ \ \ x\in B_r\backslash B_\varepsilon.
 \end{equation}
 Applying the Alexandroff-Bakelman-Pucci estimate to $D_{i}(a_{ij}D_ju_r)=0$ in $B_2$, we have
 $$\|u_r\|_{L^{\infty}(B_2)}\leq \|u\|_{L^{\infty}(\partial B_2)}+\|\bar{C} E\|_{L^{\infty}(\partial B_2)}+\|w\|_{L^{\infty}(B_2)},$$
 where $C>0$ depends only on $n$, $\lambda$ and $\Lambda$.
Hence we prove that $u_r$ is  bounded in compact sets  of $\mathbb{R}^n$.
From the H\"{o}lder estimate, it follows that there exists a subsequence of $\{u_r\}_{r=1}^\infty$ that converges uniformly to a function $v\in C(\mathbb{R}^n)$ in compact sets  of $\mathbb{R}^n$. Furthermore, by Caccioppoli inequality (see \cite{FR2022,GT83}), we have, for any compact $K\subset \mathbb{R}^n$,
$$\int_{K}|D(u_{r_1}-u_{r_2})|^2\,dx \rightarrow 0 \ \ \mbox{as}\ \ r_1, r_2 \rightarrow\infty.$$
Consequently, $v\in W^{1,2}_{loc}(\mathbb{R}^n)$ and satisfies
$$D_i(a_{ij}D_jv)=0\ \ \mbox{in}\ \ \ \mathbb{R}^n$$
with $$|B_r|^{-1} \int_{B_r}v^2\,dx \leq Cr^{2N}$$
for all $ r \geq 1$ and some  integer $N\geq 0$ and constant $C>0$. Hence by Avellaneda-Lin's result, we obtain
$$v(x)=\sum_{|\nu|\leq N} p_\nu(x) x^{\nu},$$
where $p_\nu(x)$ are periodic and H\"{o}lder continuous.

Finally, letting $r\rightarrow \infty$ in (\ref{eq2}), we have
$$  |v(x)-u(x)|\leq \bar{C} E(x)+w(x),\ \ \ x\in \mathbb{R}^n\setminus \overline{B_1}.$$
Hence, by the aid of  Serrin and Weinberger's result \cite[Theorem 5]{SW1966}, we obtain the desired result.
\end{proof}

\section{The proof of Theorem \ref{th2}}

\noindent

\begin{proof}[Proof of Theorem \ref{th2}]
Without loss of generality, we assume that $\Omega$ contains the origin and we can extend $\varphi$ smoothly in  $\mathbb{R}^n$. Let $r_0=\mbox{daim}\ \Omega+1$.  In view of Lax-Milgram theorem, let $v\in \mathbb{T}\cap W^{1,2}(\mathbb{R}^n)$ be a weak solution of $D_{i}(a_{ij}D_j(b\cdot x+v(x)))=0$ in $\mathbb{R}^n$.
We set $w:=b\cdot x+v(x)$ and $\bar{C}:=\|w-\varphi\|_{L^{\infty}(\partial \Omega)}$.

For $r>r_0$, we introduce
 $$\bar{\varphi}=(1-\eta)w +\eta \varphi,$$
 where $0\leq \eta \leq 1$ be in $C^{\infty}(\mathbb{R}^n)$ with compact support in $B_{r_0}$ such that $\eta=1$ in $\Omega$.

Let $u_r\in W^{1,2}(B_r)$ be a  weak solution of
\begin{equation*}
   \left\{
\begin{aligned}
   &D_{i}(a_{ij}D_ju_r)=0\ \ \mbox{in}\ \ B_r \backslash \overline{\Omega},\\
&u_r=\bar{\varphi}\ \  \mbox{on}\ \ \partial (B_r \backslash \overline{\Omega}).
\end{aligned}
\right.
 \end{equation*}
 Clearly, we see that
 \begin{equation*}
   \left\{
\begin{aligned}
   &D_{i}(a_{ij}D_ju_r)=0\ \ \mbox{in}\ \ B_r \backslash \overline{\Omega},\\
    &D_{i}(a_{ij}D_jw)=0\ \ \mbox{in}\ \ B_r \backslash \overline{\Omega},\\
&w-\bar{C} \leq \bar{\varphi}=u_r \leq w+\bar{C} \ \  \mbox{on}\ \ \partial (B_r \backslash \overline{\Omega}).
\end{aligned}
\right.
 \end{equation*}
 Thus, applying the comparison principle for  weak solutions yields
 $$w-\bar{C}\leq u_r\leq w+\bar{C} \ \ \mbox{in}\ \ B_r \backslash \overline{\Omega}.$$
From the H\"{o}lder estimate, it follows that there exists a subsequence of $\{u_r\}_{r=1}^\infty$ that converges uniformly to a function $u\in C(\mathbb{R}^n \backslash \overline{\Omega})$ in compact sets  of $\mathbb{R}^n \backslash \overline{\Omega}$. Furthermore, by Caccioppoli inequality, we know that $u\in W^{1,2}_{loc}(\mathbb{R}^n \backslash \overline{\Omega})$ and satisfies
$$D_i(a_{ij}D_ju)=0\ \ \mbox{in}\ \ \ \mathbb{R}^n \backslash \overline{\Omega}$$
with
$$w-\bar{C}\leq u\leq w+\bar{C} \ \ \mbox{in}\ \ \mathbb{R}^n \backslash \overline{\Omega}.$$

We are now in a position to show that $u$ is continuous up to $\partial \Omega$ and coincides with $\varphi$ on $\partial \Omega$. More precisely, for any $x_0\in \partial \Omega$, then we have
$$\lim_{x\in \mathbb{R}^n \backslash \overline{\Omega},\ x\rightarrow x_0}u(x)=\varphi(x_0).$$
Since $\Omega$ satisfies an exterior cone condition, by \cite[Corollary (9.1)]{LSW1963}, any point $x_0\in \Omega$ is regular. Moreover, by \cite[Lemma (3.2)]{LSW1963}, there exists a barrier at $x_0$ (see for the definition of a barrier \cite[Definition (3.2)]{LSW1963} ). By the standard procedure, we can prove the claim.

Now we prove $\lim_{|x|\rightarrow \infty} (u-w)(x)$ exists. Let $c:=\liminf_{|x|\rightarrow \infty}(u-w)(x)<\infty$. Then for any $\epsilon>0$, there exist sequences of $\{R_i\}_{i=1}^{\infty}\subset \mathbb{R}^{+}$ increasing and $\{x_i\}_{i=1}^\infty$ with $R_1\geq r_0$ and $x_i\in \partial B_{R_i}$ such that
$$u(x)-w(x)-c+\epsilon\geq 0,\ \ \ x\in \mathbb{R}^n \backslash \overline{B_{R_1}}$$
and
$$u(x_i)-w(x_i)-c-\epsilon\leq 0.$$
Applying the Harnack inequality to $u-w$ in $\partial B_{R_i}$ $(i\geq 1)$, we have
$$u(x)-w(x)-c+\epsilon \leq C (u(x_i)-w(x_i)-c+\epsilon) \leq 2C\epsilon,\ \ \ x\in \partial B_{R_i},$$
where $C$ depends only on $\lambda$, $\Lambda$ and $n$.

For any $x\in \mathbb{R}^n \backslash \overline{B_{R_1}}$, there exists some $i\geq 2$ such that $x\in B_{R_{i}} \backslash \overline{B_{R_1}}$. Therefore, by the comparison principle in $ B_{R_{i}} \backslash \overline{B_{R_1}}$, we obtain
$$u(x)-w(x)-c+\epsilon  \leq 2C\epsilon,\ \ \ x\in B_{R_{i}} \backslash \overline{B_{R_1}},$$
which implies that
$$-\epsilon \leq u(x)-w(x)-c  \leq C\epsilon,\ \ \ x\in \mathbb{R}^n \backslash \overline{B_{R_1}}.$$
Letting $\epsilon \rightarrow 0$, we obtain $\lim_{|x|\rightarrow \infty} (u-w)(x)=c$.

Finally, we recall the Green function $E(x)=G(x,0)$.
Fix any  $x\in \mathbb{R}^n \backslash \overline{\Omega}$. For any  $\epsilon>0$, we choose some constants $C>0$ and $R>r_0$ such that $x\in B_{R}$
and we have
$$ |u-w-c|\leq CE+\epsilon\ \ \ \mbox{on}\ \ \ \partial B_{r_0}\cup \partial B_{R}.$$
From the comparison principle, it follows that
$$|u(x)-w(x)-c|\leq CE+\epsilon.$$
Letting $\epsilon \rightarrow 0$, we obtain the desired result.
\end{proof}

\end{document}